\documentclass{amsart}
\usepackage{amssymb}
\usepackage{amsfonts}

\newtheorem{theorem}{Theorem}
\theoremstyle{plain}

\newtheorem{proposition}{Proposition}
\newtheorem{remark}{Remark}

\numberwithin{equation}{section}
\input{tcilatex}

\newcommand{\R}{{\mathbb{R}}}
\newcommand{\C}{{\mathbb{C}}}

\begin{document}
\title[Four operations of fractions and their relations]{A class of fractions adding how pupils wish\\
}
\author{Martin Himmel}
\email{martin.himmel@gmail.com}

\begin{abstract}
Motivated by children inventing creatively laws on how to add or multiply fractions,
 we introduce four fraction operations and discover some of their properties.
\end{abstract}

\maketitle

\QTP{Body Math}
$\bigskip $

\section{Introduction}
Let $x,y$ be two fractions, 
i.e., there are integers $a,b, \alpha, \beta$, those referring to the denominators $b, \beta$ different from zero, such that
$x=\frac{a}{b}$, $y=\frac{\alpha}{\beta}$.
Addition of fractions is defined in this way:
\begin{align}
\frac{a}{b} + \frac{\alpha}{\beta} := \frac{a \beta + \alpha b}{b \beta},
\label{addition}
\end{align} 
whereas multiplication reads as
\begin{align}
\frac{a}{b} \cdot \frac{\alpha}{\beta} := \frac{a  \alpha }{b \beta}
\label{multiplication}.
\end{align} 

For pupils at school it is not easy to remember those rules 
and teachers usually see plenty of variations of these laws during their professional work life
three of which we take a closer look in this paper.

For some pedagogigal remarks on fraction operation see \cite{GaDi}, \cite{PuIs} and \cite{SaPu}.

\subsection*{Dual Multiplication}
The first comes from the situation when someone wishes to add fractions in the same manner as they are
multiplied,
namely
\begin{align*}
	\frac{a}{b}  \odot \frac{\alpha}{\beta} := \frac{a  + \alpha }{b + \beta}.
\end{align*} 
We will refer to this binary operation as dual mulplication
 since it is obtained from the multiplication law \eqref{multiplication} when $\cdot$ is replaced by $+$.
 Moreover, we have $x \oplus x =x$ for all fractions $x$, thus $\oplus$ is an idempotent operation.

 \subsection*{Dual Addition}
 The second two operations are motivated similarly from the addition law \eqref{addition}
 on replacing $+$ by $\cdot$ and vice versa.
 Thus,
 \begin{align*}
 	\frac{a}{b}  \oplus \frac{\alpha}{\beta} :=  \frac{a + \beta \cdot \alpha + b}{b + \beta}.
 \end{align*} 
We will refer to this binary operation $\oplus$ as dual addition of first kind,
 since it is obtained from the addition law \eqref{addition} when $\cdot$ is exchanged with $+$ and vice versa.
 
 Similarly, we call 
  \begin{align*}
 	\frac{a}{b}  \boxplus  \frac{\alpha}{\beta} := \frac{ \left( a + \beta \right) \cdot \left(\alpha + b \right)}{b + \beta}.
 \end{align*} 
the dual fraction sum of second kind.
In a nutshel, the operation $\boxplus$ is obtained from $\oplus$ when using the former way of putting brackets (namely the inherited associativity by \eqref{addition}).
In contrast to  the usual operations for fractions  together with dual multiplication
the here newly introduced additions do not always commute.
More precisely, for fractions $x=\frac{a}{b}$ and $y=\frac{\alpha}{\beta}$ we have:
 $x \boxplus y=y \boxplus x$ 
iff $a (b-1)-b=\alpha (\beta -1) - \beta$,
and $x \oplus y=y \oplus x$ iff $\alpha + \beta (\alpha -1)= a +b (a-1)$.

Thus, in the case 
\begin{itemize}
	\item
	$b=\beta=1$ where $x$ and $y$ are integers: dual addition of second kind commutes iff $a=\alpha$;
	\item
	$a=\alpha=1$ where $x$ and $y$ are Barbilonian fractions: $\oplus$ commutes iff $b=\beta$.
\end{itemize}

It is notweworthy that the dual sum of second kind of 'zero fractions' equals the harmonic mean of their denominators, i.e.
\begin{align*}
\frac{0}{b}	\boxplus \frac{0}{\beta}=\frac{1}{\frac{1}{b} + \frac{1}{\beta}}
= H(b,\beta),
\end{align*}
where $H: (0, +\infty)^2 \to (0, +\infty)$ defined by $H(b,\beta)=\frac{2 b \beta}{b+\beta}$ is the harmonic mean.

Unsurprisingly  these newly introduced operations do not always agree with the usual operations of fractions, namely with addition and multiplication, respectively.
Nevertheless,  it is of great interest to know for which classes of fractions these operationss give the same result.
More concretely, in this note
 we answer the question for which fractions it holds 
 $x+y=x  \cdot y$, $x+y=x \oplus y$ or $x+y=x \boxplus y$, respectively,
or, more generally, when they are proportional.

\section{Four operations of fractions and their relations}

It would be appreciated very much by most pupils if addition of fractions would be defined 'component wisely' and
thus
in the following way:
\begin{equation*}
	\frac{a}{b} \odot \frac{\alpha}{\beta}:=\frac{a + \alpha }{b + \beta}.
\end{equation*}
We will refer to this operation as dual multiplication or pupil fraction addition%
\footnote{As a mathematician one has to check that the here newly defined operation $\odot$ is indeed well-defined
Since in a fraction numerator and denominator are not uniquely defined, because $\frac{a}{b}=\frac{at}{bt}$ for any non-zero $t$ (Kuezbarkeit), this property - called zero homogeneity - must be inherited by the here introduced operation. The reader may verify that $\frac{ta}{tb} \odot \frac{s \alpha }{s \beta} = \frac{a}{b} \odot \frac{\alpha }{\beta}$ iff $s=t$ or $\frac{a}{b} = \frac{\alpha }{\beta}$ ("Pupil sum is only diagonaly well-defined.").}
.

In general, of course, this so defined addition does not agree with usual addition of fractions, 
but it is of high interest to know
when both addition laws agree, or, more generally, when they are proportional.

\subsection{Addition}
We deal with this problem in the following
\begin{proposition}
	Let  $a,b, \alpha, \beta \in \R$, $b \beta \neq 0$. Then we have:
	\begin{equation*}
		\frac{a}{b} \odot \frac{\alpha}{\beta}=	\frac{a}{b} + \frac{\alpha}{\beta}
	\end{equation*}	
	if, and only if,
	\begin{equation*}
		\alpha b^2  + a \beta^2 = 0.
	\end{equation*}	
\end{proposition}

\begin{proof}
	Assume that pupil addition gives the same result as usual addition,
	namely 
	\begin{equation*}
		\frac{a+\alpha}{b+\beta}=\frac{a \beta + \alpha b}{b \beta}.
	\end{equation*}
	Multiplying here both sides by $b \beta (b+\beta)$ yields
	\begin{equation*}
		\left(a+\alpha \right) b \beta = \left( a \beta + \alpha b\right) \left( b+\beta \right), 
	\end{equation*}
	and thus
	\begin{equation*}
		a b \beta+\alpha b \beta = a \beta b + \alpha b^2  + a \beta^2 + \alpha b \beta
	\end{equation*}
	simplifying to
	\begin{equation*}
		\alpha b^2  + a \beta^2 = 0.
	\end{equation*}
Vice versa, let us assume this relation.
	Since $b$ and $\beta$ are denominators of fractions,
	  both $b$ and $\beta$ are different from zero.
	Hence,
	\begin{equation*}
		\frac{a}{b}=\frac{-b}{\beta}\frac{\alpha}{\beta},
	\end{equation*}
	which means that the first fraction is the $\frac{-b}{\beta}$-multiple of the second.
	We have
	\begin{eqnarray*}
	\frac{a}{b}+\frac{\alpha}{\beta}
	&=&
	\frac{-b}{\beta}\frac{\alpha}{\beta}+\frac{\alpha \beta}{\beta^2}\\
	&=&
		\frac{\alpha (\beta -b)}{\beta^2}\\
		&=&
	\frac{a}{b} \odot \frac{\alpha}{\beta}	\\
	&=&
\frac{-b}{\beta}\frac{\alpha}{\beta}\odot \frac{\alpha \beta}{\beta^2}	\\
	&=&
	\frac{\alpha (\beta -b)}{\beta^2}.
	\end{eqnarray*}
This concludes the proof.
\end{proof}

For sake of better memorization we note 
\begin{remark}
	Sum and dual product ( $=$ pupil sum) give the same result for fractions 
	whose numerators $a$ and $\alpha$ lie on the decreasing line $a=\left(i \frac{b}{\beta} \right)^2 \alpha$ where $i^2=-1$.
\end{remark}

On the other hand, we ask for sake of curiosity when addition gives the same result as multiplication.
\begin{proposition}
	Let  $a,b, \alpha, \beta \in \R$, $b \beta \neq 0$. The equality 
	\begin{equation*}
		\frac{a}{b} + \frac{\alpha}{\beta}=	\frac{a}{b} \cdot \frac{\alpha}{\beta}
	\end{equation*}	
	holds true if, and only if,
	\begin{equation*}
		\frac{a}{b} =\frac{\alpha }{ \alpha - \beta}.
	\end{equation*}	
\end{proposition}

\begin{proof}
	Assume that addition gives the same result as multiplication,
	namely 
	\begin{equation*}
		\frac{a\beta +b \alpha}{b \beta}= \frac{a \alpha}{b \beta}.
	\end{equation*}
	Multiplying both sides by ${b \beta}$ gives us
	\begin{equation*}
		a\beta +b \alpha= a \alpha,
	\end{equation*}
	and thus
	\begin{equation*}
		a(\beta - \alpha) +b \alpha=0.
	\end{equation*}
	If $\alpha = \beta$, we get $b \alpha=0$ implying $\alpha=0$ since $b \neq 0$ by assumption. 
	But this would imply that also $\beta=0$, which is obviously excluded.
	Hence, we may assume without loss of generality that $\alpha \neq \beta$.
	We obtain
	\begin{equation*}
		\frac{a}{b} =\frac{\alpha }{ \alpha - \beta}.
	\end{equation*}
To prove the converse, let us assume this relation.
We have

\begin{eqnarray*}
	\frac{a}{b}+\frac{\alpha}{\beta}
	&=&
	\frac{\alpha }{ \alpha - \beta}+\frac{\alpha }{\beta}\\
	&=&
   \frac{\alpha \beta }{ (\alpha - \beta) \beta}+\frac{\alpha (\alpha - \beta) }{\beta (\alpha-\beta)}\\
   &=&
   \frac{\alpha^2 }{(\alpha-\beta) \beta }\\
	&=&
	\frac{a}{b} \cdot \frac{\alpha}{\beta}	\\	
	&=&
	\frac{\alpha }{ \alpha - \beta} \cdot \frac{\alpha}{\beta},
\end{eqnarray*}
which completes the proof.
\end{proof}

Note also that we may also answer in general when addition agrees with multiplication (not just for fractions):
assuming $x+y=xy$ we get $y (x-1)=x$. For $x=1$ this holds true. 
Hence, assuming $x \neq 1$, we have $y=\frac{x}{x-1}$ 
(and, obviously, $\frac{x}{x-1} = \sum_{k=0}^\infty {\frac{1}{x^k}}$ outside the unit circle).

When addition agrees with sum of second kind is considered in the next
\begin{proposition}
	Let $x=\frac{a}{b}$, $y=\frac{\alpha}{\beta}$ be fractions. Then we have $x+y= x \boxplus y$
		if, and only if,
		\begin{align*}
		b^2 \left(a \beta  - \alpha +\frac{1}{2} \beta ^2 \right) + \beta^2 \left(\alpha b  - a +\frac{1}{2} b ^2 \right) 
		+ ab \beta (\alpha -1) 
		- \alpha \beta b=0.	
	\end{align*}
	\end{proposition}

\begin{proof}
Let us assume that, for two fractions $x,y$, their sum equals their dual sum of second kind:
$x+y= x \boxplus y$. By the definition of $\boxplus$ this means 
		\begin{align*}
		\frac{a \beta + \alpha b }{b \beta}=\frac{(a+\beta)(b+\alpha)}{ b+\beta}.
	\end{align*}
Clearing the denominators we obtain
\begin{align*}
	(a \beta + \alpha b) (b+\beta)=(a+\beta)(b+\alpha) b \beta,
\end{align*}
we get
\begin{align*}
(ab +a\alpha + b \beta + \alpha  \beta) b\beta = ab \beta  + a \beta ^2 + b^2 \alpha  + b \alpha \beta
\end{align*}
and thus
\begin{align*}
a b^2 \beta + a \alpha b \beta + (\beta b)^2 + \alpha \beta^2 b=ab \beta + a \beta^2 + b^2 \alpha + b \alpha \beta,
\end{align*}
which yields the claim. The conversion is easy to verify.
\end{proof}

\subsection{Multiplication}
In this setion we answer when fraction multiplication agrees with one of the four other operations.

When the product equals dual product is answered in the following
\begin{proposition}
	Let $x=\frac{a}{b}$, $y=\frac{\alpha}{\beta}$ be fractions. Then it holds $xy= x \odot y$
	if, and only if,
	\begin{align*}
	a b (\alpha  - \beta) =\alpha \beta (b-a).
	\end{align*}
\end{proposition}

\begin{proof}
	Let us take two fractions $x=\frac{a}{b}$, $y=\frac{\alpha}{\beta}$ whose
	product equals their dual product: $xy= x \odot y$.
	By the definition of  $\odot$ this means
	\begin{align*}
		\frac{a \alpha }{b \beta}=\frac{a+\alpha}{ b+\beta}.
	\end{align*}
Hence,
\begin{align*}
	a \alpha b + a \alpha \beta=a b\beta+\alpha b\beta,
\end{align*}
and thus
\begin{align*}
	a b (\alpha  - \beta) =\alpha \beta (b-a).
\end{align*}
If $\alpha = \beta$, we get $0=\alpha^2 (b-a)$, whence also $a=b$.
In this situation both fractions are equal to one and everything is fine.
The case where exactly one fraction equals one, but the other not, cannot occur.
Therefore, we assume without loss of generality that  $a \neq b$ and $\alpha \neq \beta$.
Hence,
\begin{align*}
	\frac{a b}{b-a}  =\frac{\alpha \beta}{\alpha  - \beta}
\end{align*}
is constant, say $=:c$, to yield
\begin{align*}
a b =c ({b-a}).
\end{align*}
Consequently,
\begin{align*}
	a (b+c) =c b,
\end{align*}
and, since $c \neq -b$,
\begin{align*}
	a  =\frac{c}{b+c} b.
\end{align*}
The converse implication is easy to verify.
	\end{proof}

When product and dual sum of second kind agree is answered in the following
\begin{proposition}
	Let $x=\frac{a}{b}$, $y=\frac{\alpha}{\beta}$ be fractions. Then it holds $xy= x \boxplus y$
if, and only if,
\begin{align*}
b^2 (a + \beta) + \beta^2 (b+\alpha) + \alpha \beta b + ab \beta =0.
\end{align*}
\end{proposition}

Similary, the question when product and dual sum of first kind are equal, is not difficult to answer.
\begin{proposition}
	Let $x=\frac{a}{b}$, $y=\frac{\alpha}{\beta}$ be fractions. Then, it holds $xy= x \oplus y$, i.e.
	\begin{align*}
		\frac{a \alpha }{b \beta}=\frac{a+\beta b+\alpha}{ b+\beta},
	\end{align*}
	if, and only if,
	\begin{align*}
		ab (\alpha -\beta ) + \alpha \beta (a-b) +  =(\beta b)^2.
	\end{align*}
\end{proposition}

\section{The case $\lambda=0$ }
Here we describe which fractions 'add' to zero. This corresponds to the case of kernels of linear operators.
Which fractions add or multiply to zero, is described in the following
\begin{theorem}
	Let $x=\frac{a}{b}$, $y=\frac{\alpha}{\beta}$ be fractions. It holds:
	\begin{enumerate}
		\item[(i)]
		 $x+y= 0$ iff $y=-x$;
		 \item[(ii)]
		 $x y=0$ iff $x=0$ or $y=0$;
		 \item[(iii)]
		 $x \oplus y= 0$ iff $a+\beta \alpha +b=0$, i.e. $y=-\frac{b}{\beta^2} x -\frac{b}{\beta^2}$;
		 \item[(iv)]
		 $x \boxplus y= 0$ iff $(a+\beta)(\alpha + b)=0$, i.e. $a=-\beta$ or $\alpha=-b$.
	\end{enumerate}
\end{theorem}

\begin{proof}
The statements $(i)$ and $(ii)$ on sums and products, respectively, are obvious.
To prove $(iii)$, let $ x \oplus y= 0$, i.e. $\frac{a+\beta \alpha + b}{b+ \beta}=0$. 
Clearing the denominator, we obtain $a+\beta \alpha +b=0$, whence $a +b=-\alpha \beta $,
consequently $x+1=-\frac{\beta^2}{b} y$ and thus $y=-\frac{b}{\beta^2} x -\frac{b}{\beta^2}$.\\
Statement  $(iv)$ follows from the definition of $\boxplus$.
The converse implications are easily verified.
\end{proof}
\section{Linear dependence - $\lambda \notin \{0,1\} $ }

Analogously to the previous sections one can prove the following

\begin{theorem}\label{linDependence}
	Let $x=\frac{a}{b}$, $y=\frac{\alpha}{\beta}$ be fractions and $\lambda \in \C \setminus \{0,1\}$ be fixed. It holds:
\begin{enumerate}
	\item
	$x+y= \lambda x \odot y $ iff $(a b \beta  + \alpha \beta b)(1-\lambda)+ a \beta^2 + \alpha b^2=0$;
	\item
	$x+y=\lambda x y$ iff $y=\frac{x}{\lambda x -1}$ for $x \neq \frac{1}{\lambda}$
	\item
	$x+y=\lambda x \boxplus y$ iff $ab \beta (1-\lambda) + a \beta ^2 + \alpha \beta b (1-\beta) + b^2 (\alpha - \lambda \beta ) $.
	\item
	$xy= \lambda x \odot y$ iff $ab (\alpha - \lambda \beta) + \alpha \beta  (a-\lambda b)=0$.
	\item
	$xy= \lambda x \boxplus y$ iff $a \alpha \left( b (1 -\lambda \beta) +  \beta \left(1 - \lambda \frac{b \beta}{a} \right) \right) - \lambda (a b^2 \beta + (b \beta)^2)=0$;
	\item
	$xy= \lambda x \oplus y$ iff $ab (\alpha - \lambda \beta) + \alpha \beta (a - \lambda b \beta)= \lambda b^2 \beta$.
\end{enumerate}
\end{theorem}

\section{Homogeneity}
In this section we investigate how the five operations between fraction behave under scaling of their inputs.

\begin{align}
s \frac{a}{b} +t \frac{\alpha}{\beta} = \frac{s a \beta + t \alpha b}{b \beta}
= s  \frac{a \beta + \frac{t}{s} \alpha b}{b \beta}.
\end{align} 
For $s=t$ we have $t \frac{a}{b} +t \frac{\alpha}{\beta}=t \frac{a \beta + \alpha b}{b \beta}$.
Scaling of the inputs results in a sum scaled by the same amount, which is, of course, a consequence of linearity.
We describe this phenomennon as: the operation $+$ is $1$-homogeneous.

Secondly, multiplication: for $s,t \neq 0$ we have
\begin{align}
s \frac{a}{b} \cdot t \frac{\alpha}{\beta} = st \frac{a  \alpha }{b \beta}.
\end{align} 
Thus, $t \frac{a}{b} \cdot  t\frac{\alpha}{\beta} = t^2 \frac{a  \alpha }{b \beta}$. 
So multiplication of fractions is positively homogeneous of degree $2$.

Thirdly, for dual multiplication it holds
\begin{align*}
s\frac{a}{b}  \odot t\frac{\alpha}{\beta} = \frac{sa  + t\alpha }{b + \beta}
=s \frac{a  + \frac{t}{s} \alpha }{b + \beta}.
\end{align*} 
For $s=t$ hence $t\frac{a}{b}  \odot t\frac{\alpha}{\beta} = t \frac{a  + \alpha }{b + \beta}$,
which means that dual multiplication $\odot$  is a $1$-homogeneous operation.

Fourthly, dual addition of first kind:
\begin{align*}
s \frac{a}{b}  \oplus t \frac{\alpha}{\beta} =  \frac{s a + \beta \cdot t \alpha + b}{b + \beta}
=  s \frac{a + \beta \cdot \frac{t}{s} \alpha + \frac{1}{s}b}{b + \beta}.
\end{align*} 
For $s=t$ we get $t \frac{a}{b}  \oplus t \frac{\alpha}{\beta} = t \frac{a + \beta \cdot  \alpha + \frac{1}{t}b}{b + \beta}$.
Thus, in general dual addition of first kind $\oplus$  is not a homogeneous operation.

Sixthly, dual addition of second kind:
\begin{align*}
s \frac{a}{b}  \boxplus  t \frac{\alpha}{\beta} = \frac{ \left(s a + \beta \right) \cdot \left(t \alpha + b \right)}{b + \beta}
=s  \frac{ \left(a + \frac{1}{s} \beta \right) \cdot \left(\frac{t}{s} \alpha + \frac{1}{s}b \right)}{b + \beta}.
\end{align*} 
For $s=t$ we have $t \frac{a}{b}  \boxplus  t \frac{\alpha}{\beta} =
t  \frac{ \left(a + \frac{1}{t} \beta \right) \cdot \left( \alpha + \frac{1}{t}b \right)}{b + \beta}$
So in general neither dual addition of second kind $\oplus$ is  a homogeneous operation.

\section{Well-Definedness of the Fraction Operations}
When defining equality of two fractions $x=\frac{a}{b}$ and $y=\frac{\alpha}{\beta}$ it is not zielfuehrend 
to do this in a 'point wise' manner
\footnote{By this we mean:  $\frac{a}{b}=\frac{\alpha}{\beta}$ if $a=\alpha$ and $b=\beta$.}  
because numerator and denominator are uniquely determined up to scaling with a non-zero number
(i.e. for any fraction $x=\frac{a}{b}$ we have: $\frac{a}{b}=\frac{at}{bt}$ for all $t \in \R, t \neq 0$.)
Rather two fractions $x=\frac{a}{b}$ and $y=\frac{\alpha}{\beta}$ are called equal, if
$a \beta-b \alpha=0$.
Long story short, when introducing some operartions between fractions based on a concrete representation of
a fraction we need to verify that the result is unique and hence does not depend on the representation of the fraction we chose,
so we have indeed a well-defined operation.

To illustrate this conept, we check well-definedness of addition and mulitiplication, respectively.

\subsection{Well-Definedness of Fraction Addition}
\begin{remark}
For two fractions $x=\frac{a}{b}$ and $y=\frac{\alpha}{\beta}$ and arbitrary $s,t \in \R$, $st \neq 0$, we have:
\begin{align*}
	\frac{s a}{s b}	+ \frac{t \alpha}{t \beta} &=&\frac{s a t \beta + t \alpha s b}{s b t \beta}\\
	&=&\frac{st (a  \beta +  \alpha  b) }{st b  \beta}\\
	&=&\frac{a  \beta +  \alpha  b }{b  \beta}
	=\frac{a}{ b}	+ \frac{ \alpha}{ \beta}.
\end{align*}	
Thus, the sum of two fractions does not depend on the specific representation of the two fractions 
and so fraction addion is well-defined. 
\end{remark}

\subsection{Well-Definedness of Fraction Multiplication}
\begin{remark}
	For two fractions $x=\frac{a}{b}$ and $y=\frac{\alpha}{\beta}$ and arbitrary $s,t \in \R$, $st \neq 0$, we have:
	\begin{eqnarray*}
		\frac{s a}{s b}	\cdot \frac{t \alpha}{t \beta} &=&\frac{s a t \alpha}{s b t \beta}	\\
		&=&\frac{st (a  \alpha)}{st (b  \beta)}	=\frac{a}{ b}	\cdot \frac{ \alpha}{ \beta}.
	\end{eqnarray*}	
	Thus, the product of two fractions does not depend on the specific representation of the fractions 
	and so fraction multiplication is well-defined. 
\end{remark}

\subsection{Well-Definedness of Dual Fraction Addition of First Kind}
There is an intriguing thing about dual fraction addition of first kind which stems from fact 
that the sign of a fraction can be interpreted as the sign of the denominator or as the sign of the numerator.
For instance, we have
\begin{align*}
\frac{1}{2} \oplus \frac{-1}{2} = \frac{1+2+(-1)*(2)}{2+2} =\frac{1}{4},
\end{align*}	
but on the other hand
\begin{align*}
\frac{1}{2} \oplus \frac{1}{-2} = \frac{1+2+(1)*(-2)}{2-2} =\frac{1}{0}.
\end{align*}
With that in mind well-definedness of the $\oplus$-operation is not just a mathematical formality,
but non-trivial and a little involved.
In constrast to the usual addition and multiplication of fractions, it does not lead to an algebraic identity,
but rather to a condition on the numerators and denominators of the fractions.


\begin{remark}
The dual sum of first kind of the fractions  $x=\frac{a}{b}$ and $y=\frac{\alpha}{\beta}$ is well-defined if, and only if,
\begin{align}
	s \left(\alpha b (1+\beta) - a (b+\beta) +b^2 \right) 
	+t \left(b (\beta-\alpha) + \alpha \beta^2 \right)
	= t^2   \alpha \beta  \left( 1+ \beta \right)
	\label{eq:wd_ds1}
\end{align}	
for all $s,t \in \R$, $st \neq 0$.
\end{remark}
\begin{proof}
By the definition of dual sum of first kind, for two fractions $x=\frac{a}{b}$ and $y=\frac{\alpha}{\beta}$ and arbitrary 
$s,t \in \R$, $st \neq 0$, the well-definedness condition 
	$\frac{sa}{sb} \oplus \frac{t\alpha}{t\beta} = \frac{a}{b} \oplus \frac{\alpha}{\beta}$ amounts to
\begin{align*}
	\frac{a + \beta \alpha+ b}{b+\beta} = \frac{sa + t^2 \beta \alpha+ sb}{s b+t \beta}
\end{align*}
and hence	\eqref{eq:wd_ds1} as claimed.

\end{proof} 
Since this condition holds for all $s,t \in \R$, $st \neq 0$, the terms in front of $s,t$ and $t^2$ must vanish.
Hence, we get the following system of equations:
\begin{eqnarray*}
	b^2 + \alpha b (1+\beta) = a (b+\beta)\\
	b (\alpha - \beta) = \alpha \beta^2\\
	\alpha \beta (1+ \beta)=0.
\end{eqnarray*}
Obviously, 
we have 
$b \beta  \neq 0$.
To solve the third equation, assume $\alpha =0$ (which means that the second fraction is zero).
From the second equation we obtain $\alpha = - \frac{b}{\beta}$, a contradiction.
Thus, we may assume without lost of generality $\alpha \neq 0$.
To satisfy the third equation, we conclude $\beta=-1$.
The second equation then reads 

\begin{align*}
	b (\alpha +1)=\alpha.
\end{align*}
We may devide here by $\alpha+1$,
since $\alpha =-1$ together with this equation would yield a contradiction to our assumption $\alpha \neq 0$,
resulting in
\begin{align*}
	b =\frac{\alpha}{\alpha +1}.
\end{align*}
From the first equation we now obtain
\begin{align*}
	a= -\frac{\alpha^2}{\alpha+1},
\end{align*}
and thus
\begin{align*}
	\frac{a}{b} \oplus \frac{\alpha}{\beta} &=& \frac{\left(-\frac{\alpha^2}{\alpha+1}\right)}{\left(\frac{\alpha}{\alpha+1}\right)} \oplus \frac{\alpha}{-1}\\
	&=& \frac{-\alpha}{1} \oplus \frac{-\alpha}{1}\\
	&=& \frac{-2\alpha+1}{2}
\end{align*}
Note here that again we attributed the sign of the fraction suitably to avoid a zero denominator.
On the other hand, $\frac{\alpha}{-1} \oplus \frac{\alpha}{-1}=\frac{1}{2}$.
Together with the latter relation this would give us $\alpha=0$,
which was excluded before.
When accepting here zero denominators as a natural thing,
 from $\frac{-\alpha}{1} \oplus \frac{\alpha}{-1} = \frac{-2\alpha+1}{0}$we would obtain $\alpha=\frac{1}{2}$.
We observe that on the one hand zero denominators appear naturally 
in the context of fraction addition of first kind,
and, on the other hand, some sign convention is needed.

\subsection{Well-Definedness of Dual Fraction Addition of Second Kind}

Similarly as for fraction addition of first kind one can prove the following

\begin{remark}
	The dual sum of second kind of the fractions  $x=\frac{a}{b}$ and $y=\frac{\alpha}{\beta}$ is well-defined if, and only if,
	\begin{align}
		s (a \alpha b  +ab \beta ++ a b^2+ \beta b^2)+t (a \alpha \beta + ab \beta+\alpha b^2 +b \beta^2)
		=st \left( (a \alpha)(b+\beta) +b^2 \beta + b \beta^2 \right)+ \\
		+ s^2 ab(b +\beta)+t^2 \alpha \beta (b +\beta)
	\end{align}	
	for all $s,t \in \R$, $st \neq 0$.
\end{remark}

\subsection{Well-Definedness of Dual Fraction Multiplication}
This is properly the most common operation at school 
since almost every pupil whished to add fractions in this way.

In a nutshell, the result is the following: pupil addition ($=$ dual fraction multiplication) is well-defined
only for different fractions.
The dual product of one fraction may be scaled only by the same factor.

\begin{remark}
	For two fractions $x=\frac{a}{b}$ and $y=\frac{\alpha}{\beta}$ 
	the dual product of $x$ and $y$ is well-defined if, and only if,
	\begin{align}
		(s-t) (a \beta - \alpha b)=0.
		\label{eq:wd_dp}
		\end{align}	
	for all $s,t \in \R$, $st \neq 0$.
	
\end{remark} 

\begin{proof}
By the definition of dual product, the well-definedness condition amounts to
	\begin{align*}
	\frac{s a}{s b}	\odot \frac{t \alpha}{t \beta} &=& \frac{sa + t \alpha}{sb+t \beta}	\\
	&=& \frac{a+ \alpha }{b+ \beta}\\
	=\frac{a}{ b}	\odot \frac{ \alpha}{ \beta},
\end{align*}	
for all $s,t \in \R$, $st \neq 0$,
hence
\begin{align*}
(sa + t \alpha)(b+\beta) = (a + \alpha)(sb + t\beta),
\end{align*}	
and thus \eqref{eq:wd_dp}.
If $s=t$, the condition trivially holds true.
Let us assume $s \neq t$. Then $a \beta - \alpha b=0$,
which means that both fractions coincide.
In this case $x \circ x =x$, which was already stated as the idempotency property of $\circ$
 ("All fractions coincide with their dual square.")
\end{proof}


\begin{thebibliography}{9}
	\bibitem{GaDi} R. F. M. Gagani, F. M. Diano Jr, \textit{Characterizing The Difficulty In Fraction Operation},
	 Int. j. adv. res. ISSN: 2456-9992.
	
	
	\bibitem{PuIs} Y. Purwasih, V. Istihapsari*, Afit Istiandaru, \textit{We might need a little anxiety: A finding of the students’ problem-solving on fraction operation}, 
	Int. J. Educ.,
	Vol. 1, No. 1, April 2020, pp. 13-20
	p-ISSN: 2722-2683, e-ISSN: 2722-2691, DOI: 10.12928/ijei.v1i1.2316.
	
	
	\bibitem{SaPu} S A Saskiyah and R I I Putri,
	,\textit{Mathematical representation on fraction operation for seventh-grade students using collaborative learning},	
	

\end{thebibliography}
\end{document}